\documentclass{amsart}

\newtheorem{theorem}{Theorem}[section]
\newtheorem{lemma}[theorem]{Lemma}

\newtheorem{proposition}[theorem]{Proposition}
\newtheorem{corollary}[theorem]{Corollary}

\theoremstyle{definition}
\newtheorem{definition}[theorem]{Definition}

\theoremstyle{remark}

\numberwithin{equation}{section}

\newcommand\style{\mathcal }          

\newcommand{\B}{\style{B}}
\newcommand{\M}{\style{M}}

\newcommand\A{{\style A}}
\renewcommand{\H}{\style{H}}
\newcommand{\K}{\style K}

\newcommand\osr{{\style R}}
\newcommand\oss{{\style S}}
\newcommand\ost{{\style T}}

\newcommand\omin{\otimes_{\rm min}}
\newcommand\omax{\otimes_{\rm max}}

\newcommand{\MnC}{\style M _{n}(\mathbb{C})} 

\newcommand{\ran}{\text{ran} \:}

\begin{document}

\title[Toeplitz Separability, Entanglement, and Complete Positivity]{Toeplitz Separability, Entanglement, and Complete Positivity Using Operator System Duality}

\author{Douglas Farenick}
\author{Michelle McBurney}
\address{Department of Mathematics and Statistics, University of Regina, Regina, Saskatchewan S4S 0A2, Canada}
\curraddr{}
\email{douglas.farenick@uregina.ca}
\email{mmm034@uregina.ca}
\thanks{Supported in part by the NSERC Discovery Grant program}
 

\date{}

\dedicatory{In memory of Chandler Davis}

\commby{}

\begin{abstract}
A new proof is presented of a theorem of L.~Gurvits, which states that 
the cone of positive block-Toeplitz matrices with matrix entries 
has no entangled elements. We also show that in the cone of positive
Toeplitz matrices with Toeplitz entries, entangled elements exist in all dimensions.
The proof of the Gurvits separation theorem is achieved by making use of the
structure of the operator system dual of the operator system 
$C(S^1)^{(n)}$ of $n\times n$ Toeplitz matrices
over the complex field, and by determining precisely the structure of the generators of the extremal
rays of the positive cones of the operator systems $C(S^1)^{(n)}\omin\B(\H)$ and 
$C(S^1)_{(n)}\omin\B(\H)$, where $\H$ is an arbitrary Hilbert space
and $C(S^1)_{(n)}$ is the operator system
dual of $C(S^1)^{(n)}$. 
Our approach also has the advantage of providing some 
new information concerning positive Toeplitz matrices whose entries are from $\B(\H)$ when $\H$ has infinite dimension. 
In particular,
we prove that normal positive linear maps $\psi$ on $\B(\H)$ are partially completely positive in the sense
that $\psi^{(n)}(x)$ is positive whenever $x$ is a positive $n\times n$ Toeplitz matrix with entries from $\B(\H)$.
We also establish a certain factorisation theorem for positive Toeplitz matrices (of operators), showing 
an equivalence between the Gurvits approach to separation
and an earlier approach of T.~Ando to universality.
\end{abstract}

\maketitle

 
\section{Introduction}
If $\otimes_\sigma$ is an operator system tensor product of operator systems $\osr$ and $\oss$, then an element $x$
in the positive cone of $\osr\otimes_\sigma\oss$ is said to be $\otimes_\sigma$-separable if
\[
x=\sum_{j=1}^m r_j\otimes s_j,
\]
for some positive $r_j\in\osr$ and positive $s_j\in\oss$. If a positive $x$ is not $\otimes_\sigma$-separable, then it is
said to be $\otimes_\sigma$-entangled. When the operator system tensor product $\otimes_\sigma$ is the spatial one $\omin$, then 
we simply use the terms ``separable'' and ``entangled'' to describe the two possible scenarios for positive elements of $\osr\omin\oss$.

Following \cite{connes-vansuijlekom2021,farenick2021,hekkelman2022,vansuijlekom2021}, 
the operator system of $n\times n$ Toeplitz matrices over the complex numbers 
is denoted by $C(S^1)^{(n)}$. The canonical linear 
basis for $C(S^1)^{(n)}$ is the one given by $\{r_{-n+1},\dots,r_{-1},r_0,r_1,\dots,r_{n-1}\}$, whereby
\[
r_k\;=\; \left\{
       \begin{array}{lcl}
           s^k   &:\;&      \mbox{if }k\geq0  \\
            (s^*)^k          &:\;&       \mbox{if }k < 0
           
      \end{array}
      \right\}
      \quad\mbox{ and }\quad
s= \left[\begin{array}{ccccc} 0&&&&  \\ 1&0&&&  \\ &1&0&&  \\ &&\ddots&\ddots& \\ &&&1&0 \end{array}\right] .
\]
The identity matrix 
$r_0$ serves as the Archimedean order unit 
for the operator system $C(S^1)^{(n)}$.

If $\oss$ is any complex vector space, then an
element $x=\displaystyle\sum_{\ell=-n+1}^{n-1}r_\ell\otimes\tau_\ell$ of the algebraic tensor product $C(S^1)^{(n)}\otimes\oss$
can be canonically identified with the (block Toeplitz) matrix
\begin{equation}\label{tms-intro}
x 
=\left[ \begin{array}{cccccc} 
\tau_0 & \tau_{-1} & \tau_{-2} &  \dots&  \tau_{-n+2}& \tau_{-n+1} \\
\tau_1 & \tau_0 & \tau_{-1} & \tau_{-2}& \dots & \tau_{-n+2} \\
\tau_2 & \tau_1 & \tau_0 & \tau_{-1} &\ddots&   \vdots\\
\vdots & \ddots & \ddots & \ddots &\ddots &\tau_{-2} \\
\tau_{n-2} &   &  \ddots & \ddots &\ddots & \tau_{-1} \\
\tau_{n-1} & \tau_{n-2} & \dots & \tau_{2} &\tau_{1}&\tau_0
\end{array}
\right].
\end{equation}
Our first interest in the present paper is with the case where $\oss=\M_p(\mathbb C)$, the C$^*$-algebra of $p\times p$
complex matrices. In this regard, the following theorem was communicated in \cite{gurvits2001,gurvits--burnam2002}.

\begin{theorem}[Gurvits]\label{main result} Every positive element of $C(S^1)^{(n)}\omin\M_p(\mathbb C)$ is separable.
\end{theorem} 

One of the main goals of the present paper is to provide an
alternative approach to proving Theorem \ref{main result}
by making use of recent work in
\cite{connes-vansuijlekom2021,farenick2021} on operator system duality and by considering the structure of the extremal rays
of the positive cone of $C(S^1)^{(n)}\omin\M_p(\mathbb C)$. The advantage of this approach is that it is somewhat more general in that we
consider Toeplitz matrices whose entries are bounded linear operators on Hilbert spaces of arbitrary dimension. 
A secondary
advantage of this approach is that it applies to other operator systems---in particular, it applies to the operator system dual of $C(S^1)^{(n)}$, which is denoted 
by $C(S^1)_{(n)}$ and consists of trigonometric polynomials of degree at most $n-1$.
 
The operator systems $C(S^1)^{(n)}$ and $C(S^1)_{(n)}$ are of classical importance. Yet, only recently have they been linked through duality, and through this
linkage these operator systems are receiving recent renewed attention in areas such as
spectral truncation in noncommutative geometry  
\cite{connes-vansuijlekom2021,hekkelman2022,vansuijlekom2021} and operator system tensor products \cite{farenick2021}.
The purpose of the present paper is to shed light on how this linkage relates to the factorisation 
and separability of positive Toeplitz matrices of Hilbert space operators.

Throughout this paper, $\mathbb C^p$ is considered as a $p$-dimensional Hilbert space, where the inner product on $\mathbb C^p$ is the standard
inner product. Hilbert spaces of arbitrary (finite or infinite) dimension are denoted by $\H$, and $\B(\H)$ denotes the algebra of all bounded linear
operators $x:\H\rightarrow\H$. In the case where $\H=\mathbb C^p$, then $\B(\H)=\M_p(\mathbb C)$.

\section{Duality, Purity, Nuclearity, and Universality}

Recall from \cite{choi--effros1977} that if
$\osr$ is an operator system and $\osr^d$ denotes its dual space, then $\osr^d$ is a matrix-ordered $*$-vector space in which
a matrix $\Phi=[\varphi_{ij}]_{i,j=1}^p$ of linear functionals $\varphi_{ij}:\osr\rightarrow\mathbb C$ is considered
positive whenever the linear map $r\mapsto [\varphi_{ij}(r)]_{i,j=1}^p$ is a completely positive linear map $\hat\Phi:\osr\rightarrow\M_p(\mathbb C)$.
Furthermore, if
$\phi:\osr\rightarrow\oss$ is a linear map of operator systems and $\phi^d:\oss^d\rightarrow\osr^d$ denotes the
adjoint transformation as linear mapping of matrix-ordered $*$-vector spaces, then 
$\phi$ is positive if and only if $\phi^d$ is positive. Likewise,
$\phi\otimes\mbox{\rm id}_{\M_p(\mathbb C)}$ is positive if and only if 
$\phi^d\otimes\mbox{\rm id}_{\M_p(\mathbb C)}$ is positive, for $p\in\mathbb N$.
Importantly, if $\osr$ has finite dimension, then the matrix-ordered space $\osr^d$ is in fact an operator system, 
and any faithful state on $\osr$ serves as an Archimedean order unit for the matrix ordering of $\osr^d$.

For any operator system $\osr$, the cone of positive $p\times p$ matrices over $\osr$ is denoted by $\M_p(\osr)_+$.
With respect to the operator system tensor product $\omin$, we have
\[
\M_p(\osr)_+ = \left(\osr\omin\M_p(\osr)\right)_+.
\]
Because $\M_p(\osr)_+$ is a cone, it may possess extremal rays.
(Recall that a face $\mathcal F$ in a proper convex cone $\mathcal C$ is an extremal ray
if there exists an element $\phi\in\mathcal C$, called a generator of the extremal ray,
such that
$\mathcal F= \left\{s\phi\,|\,s\in\mathbb R,\,s\geq0\right\}$.)

We shall adopt the term ``pure'' for elements in convex cones that generate extremal rays.

\begin{definition} An element $\phi$ of a 
convex cone $\mathcal C$ is \emph{pure} if the equation $\vartheta+\omega=\phi$,
for some $\vartheta, \omega\in\mathcal C$, holds only if there exists 
a scalar $s\in[0,1]$ such that $\vartheta=s\phi$ and $\omega=(1-s)\phi$.
\end{definition}

The purpose in identifying pure elements of proper convex cones $\mathcal C$ in finite-dimensional vector
spaces is that they generate, through arbitrary finite sums of pure elements, all 
elements of $\mathcal C$. 
The cones $\mathcal C$ that shall be of interest here are the cones $\M_p(\osr)_+$ and
$\mathcal C\mathcal P(\osr,\oss)$, the latter being the cone
of completely positive linear maps of the operator system $\osr$ into the operator system $\oss$.

The matrix ordering on the dual space of an operator system $\osr$ 
has the following positivity criterion for a matrix $\Phi=[\varphi_{ij}]_{i,j=1}^p\in\M_p(\osr^d)$ 
of linear functionals $\varphi_{ij}$ on $\osr$:
\begin{equation}\label{e:duality positivity}
\Phi\in \M_p(\osr^d)_+\;\mbox{ if and only if}\;\hat\Phi\in \mathcal C\mathcal P(\osr, \M_p(\mathbb C)),
\end{equation}
where $\hat\Phi:\osr\rightarrow \M_p(\mathbb C)$ is the linear map 
\[
\hat\Phi(x)= \left[ \varphi_{ij}(x)\right]_{i,j=1}^p, \mbox{ for } x\in \osr.
\]
Furthermore, when
$\osr$ has finite dimension,
the positivity condition \eqref{e:duality positivity} 
admits a more general form whereby the matrix
algebra $\M_p(\mathbb C)$ is replaced by an arbitrary operator system 
$\oss$ \cite[Section 4]{kavruk--paulsen--todorov--tomforde2011},
\cite[Lemma 8.5]{kavruk--paulsen--todorov--tomforde2013}. Specifically, if
\[
t=\sum_{j=1}^m x_j\otimes y_j\in \osr\omin\oss,
\]
and if $\hat t:\osr^d\rightarrow \oss$ is the linear map defined by
\[
\hat t(\varphi)=\sum_{j=1}^m \varphi(x_j)y_j, \mbox{ for }\varphi\in\osr^d,
\]
then $t\in (\osr\omin\oss)_+$ if and only if $\hat t$ is completely positive.
Thus, from $(\osr^d)^d=\osr$, we have
\begin{equation}\label{e:duality positivity gen}
\Phi\in (\osr^d\omin\oss)_+\;\mbox{ if and only if}\;\hat\Phi\in \mathcal C\mathcal P(\osr, \oss).
\end{equation}

Returning now to the study of positive Toeplitz matrices, 
let $C(S^1)$ denote the unital abelian C$^*$-algebra of all continuous functions $f:S^1\rightarrow \mathbb C$,
where $S^1\subset\mathbb C$ is the unit circle. For each $n\in\mathbb N$, 
$C(S^1)_{(n)}$ shall denote the vector space of those $f\in C(S^1)$ for which the Fourier coefficients $\hat f(k)$ of
$f$ satisfy $\hat f(k)=0$ for every $k\in\mathbb Z$ such that $|k|\geq n$. Hence, every $f\in C(S^1)_{(n)}$ is given by
\[
f(z)=\sum_{k=-n+1}^{n-1}\alpha_kz^k ,
\]
as a function of $z\in S^1$, 
where each $\alpha_k =\hat f(k)=\frac{1}{2\pi}\displaystyle\int_0^{2\pi}f(e^{i\theta})e^{-ik\theta}\,d\theta$. 
The vector space $C(S^1)_{(n)}$ is an 
operator system via the matrix ordering that arises from the identification
of $\M_p\left( C(S^1)_{(n)}\right)$, the space of $p\times p$ matrices with entries from $C(S^1)_{(n)}$, with the space
of continuous functions $F:S^1\rightarrow\M_p(\mathbb C)$, and where the Archimedean order unit is the canonical
one (namely, the constant function $\chi_0:S^1\rightarrow \mathbb C$ given by $\chi_0(z)=1$, for $z\in S^1$).
(The operator system of all $n\times n$ Toeplitz matrices over $\mathbb C$, which we have been denoting by
$C(S^1)^{(n)}$, has the identity matrix in $\M_{n}(\mathbb C)$ as the canonical
Archimedean order unit for $C(S^1)^{(n)}$.)

The following theorem was obtained in \cite{connes-vansuijlekom2021,farenick2021}.

\begin{theorem}[Toeplitz Duality]\label{toeplitz duality}
The  linear map $\phi:C(S^1)^{(n)}\rightarrow\left(C(S^1)_{(n)}\right)^d$ that sends a Toeplitz matrix
$t=[\tau_{k-\ell }]_{k,\ell=0}^{n-1}\in\M_{n}(\mathbb C)$ to the 
linear functional $\varphi_t:C(S^1)_{(n)} \rightarrow\mathbb C$ defined by
\begin{equation}\label{lf defn}
\varphi_t(f)=\sum_{k=-n+1}^{n-1}\tau_{-k}\hat f(k),
\end{equation}
for $f\in C(S^1)_{(n)}$, is a unital complete order isomorphism. 
That is, $C(S^1)^{(n)}$ is the operator system dual of $C(S^1)_{(n)}$.
\end{theorem}

Recall from \cite{kavruk--paulsen--todorov--tomforde2011}  that the maximal operator system tensor product, $\omax$, 
 of operator systems $\osr$ and $\ost$ is the operator system structure on the algebraic tensor product
$\osr\otimes\ost$ obtained by
declaring a matrix $x\in\M_p(\osr\otimes\ost)$ to be positive 
if, for each $\varepsilon>0$, there are $n,q\in\mathbb N$, $a\in \M_n(\osr)_+$, $b\in\M_q(\ost)_+$, and a linear map 
$\delta:\mathbb C^p\rightarrow\mathbb C^n\otimes\mathbb C^q$
such that
\[
\varepsilon(e_\osr\otimes e_\ost) + x= \delta^*(a\otimes b)\delta.
\]
(Here, $e_\osr$ and $e_\ost$ denote the Archimedean order units of $\osr$ and $\ost$.)
As sets, the inclusion $(\osr\omax\ost)_+\subseteq (\osr\omin\ost)_+$ always holds; however, 
it is typically more difficult for a matrix over $\osr\otimes\ost$ to be ``max positive'' 
than ``min positive.'' Nevertheless, there are operator systems $\osr$ for which the two
(extremal) forms of positivity coincide for all operator systems $\ost$.

\begin{definition} An operator system $\oss$ is \emph{nuclear} if
$\oss\omin\ost=\oss\omax\ost$, for every operator system $\ost$.
\end{definition}

The following matrix has a recurring role in the present paper.

\begin{definition} 
The \emph{universal positive $n\times n$ Toeplitz matrix} 
is the Toeplitz-matrix-valued function $T_n:S^1\rightarrow\M_n(\mathbb C)$ given by
\begin{equation}\label{e:univ T}
T_n(z)=\left[ \begin{array}{cccccc} 
1 & z^{-1} & z^{-2} &  \dots&  z^{-n+2}& z^{-n+1} \\
z & 1 & z^{-1} & z^{-2}& \dots & z^{-n+2} \\
z^2 & z & 1 & z^{-1} &\ddots&   \vdots\\
\vdots & \ddots & \ddots & \ddots &\ddots &z^{-2} \\
z^{n-2} &   &  \ddots & \ddots &\ddots &z^{-1} \\
z^{n-1} & z^{n-2} & \dots & z^{2} &z&1
\end{array}
\right].
\end{equation}
\end{definition}

Because the matrix $T_n(z)$ defined in \eqref{e:univ T} admits a factorisation of the form $T_n(z)=\gamma_n(z)\gamma_n(z)^*$, 
where 
\[
\gamma_n(z)=\left[\begin{array}{c}1 \\ z  \\ z^2 \\ \vdots \\ z^{n-1} \end{array}\right],
\]
the matrix-valued function
$T_n:S^1\rightarrow \M_n(\mathbb C)$ maps $S^1$ into the rank-1 positive $n\times n$ Toeplitz matrices.
The term ``universal'' is used because, for every operator system $\oss$
and positive Toeplitz matrix $x=[\tau_{k-j}]_{k,j=1}^n\in C(S^1)^{(n)}\omin\oss$, 
there exists a completely positive linear map $\phi:C(S^1)_{(n)}\rightarrow\oss$
for which $\tau_\ell=\phi(z^\ell)$, for every $\ell$ \cite{ando1970}, \cite[Theorem 7.10]{farenick2021}.

The transpose of any Toeplitz matrix is implemented by a unitary similarity transformation. Thus, we have 
\[
T_n(z^{-1}) = T_n(z)^t = u^*T_n(z)u_n,
\]
where $x\mapsto x^t$ denotes the transpose transformation and $u_n\in\M_n(\mathbb C)$ is the
unitary matrix $u_n=\displaystyle\sum_{i=1}^ne_{i,n-i+1}$ 
(where $\{e_{ij}\}_{i,j=1}^n $ are the canonical matrix units of 
$\M_n(\mathbb C)$).

We shall show in Section \S4 that $T_n$ is an entangled element of the positive cone of 
$ C(S^1)_{(n)}\omin\M_n(\mathbb C)$, 
for every $n\geq 2$, in contrast to the dual situation of 
$C(S^1)^{(n)}\omin\M_n(\mathbb C)$, which has no entangled positive elements.

\section{Separability}

The primary purpose of this section is to give a proof of the separation theorem, Theorem \ref{main result}, by determining
the structure of the pure elements of the cone $\left(C(S^1)^{(n)}\omin\B(\H)\right)_+$ for Hilbert spaces $\H$
of finite or infinite dimension. Thus, this approach addresses a more general context than that considered by Gurvits
in \cite{gurvits2001,gurvits--burnam2002}. In Section \S4, the same methodology is used in determining the
structure of the pure elements of the cone $\left(C(S^1)_{(n)}\omin\B(\H)\right)_+$, replacing $C(S^1)^{(n)}$ by its
operator system dual $C(S^1)_{(n)}$.

The following observation, inspired by \cite[Corollary 2.20]{connes-vansuijlekom2021},
is key to our approach. 

\begin{lemma}\label{pure1} If $\oss$ is any operator system, then an element
$\Phi\in \left(C(S^1)^{(n)}\omin\oss\right)_+$ is pure if and only 
if $\hat\Phi$ is a pure completely positive linear map $C(S^1)_{(n)} \rightarrow\oss$.
\end{lemma}  

\begin{proof}
Assume $\Phi\in \left(C(S^1)^{(n)}\omin\oss\right)_+$ is pure. Thus,
if 
$\Phi=\displaystyle\sum_{\ell=-n+1}^{n-1} r_\ell \otimes \tau_\ell $,
then 
$\hat\Phi=\displaystyle\sum_{\ell=-n+1}^{n-1} \varphi_{r_{\ell}} \tau_\ell$,
where the action of $\hat\Phi$ on $C(S^1)_{(n)} $ is given by
\[
\hat\Phi(f)=\sum_{\ell=-n+1}^{n-1} \hat{f}(-\ell)\tau_\ell,
\]
for every $f \in C(S^1)_{(n)}.$

Suppose that there exists completely positive maps of $\theta, \psi:C(S^1)_{(n)} \rightarrow\oss$ such that $\hat\Phi=\theta+\psi$. 
Let $\theta(z^k)=t_k$ and $\psi(z^k)=s_k$, so that
\[
\theta(f)=\theta\left(\sum_{k=-n+1}^{n-1}\hat f(k) z^k\right)=\sum_{k=-n+1}^{n-1}\hat f(k) \theta(z^k)=\sum_{k=-n+1}^{n-1} \hat{f}(k)t_k
\]
and, similarly,
\[
\psi(f)=\sum_{k=-n+1}^{n-1} \hat{f}(k)s_k,
\]
for every $f\in C(S^1)_{(n)}$. We obtain, therefore, Toeplitz matrices 
\[
T=
\begin{bmatrix}
t_0      & t_1     & \dots  & t_{n-1} \\
t_{-1}   & \ddots       &  \ddots      &    \vdots     \\
\vdots   &    \ddots   & \ddots          & t_1     \\
t_{-n+1} &   \dots  & t_{-1} & t_0     \\
\end{bmatrix}
\mbox{ and }
S=
\begin{bmatrix}
s_0      & s_1     & \dots  & s_{n-1} \\
s_{-1}   & \ddots       &  \ddots      &    \vdots     \\
\vdots   &    \ddots   & \ddots          & s_1     \\
s_{-n+1} &   \dots  & s_{-1} & s_0     \\
\end{bmatrix}.
\]
For each $k$,
\[
\tau_{-k}=\hat\Phi(z^k)=\theta(z^k)+\psi(z^k)=t_k+s_k,
\]
and
so $\tau_{-k}=t_k+s_k=a_{-k}+b_{-k}$, where $a_k=t_{-k}$ and $b_k=s_{-k}$. 
Thus, $\Phi=A+B$, for the Toeplitz matrices $A$ and $B$ determined by the elements
$a_k,b_k\in\oss$, for $k=-n+1,\dots,n-1$.
The way in which $A$ and $B$ are defined leads $\hat{A}=\theta$ and $\hat{B}=\psi$;
thus, $A$ and $B$ are positive in $C(S^1)^{(n)}\omin\oss$. Hence, 
by the purity of $\Phi$, 
$A=s\Phi$ and $B=(1-s)\Phi$, for some $s\in [0,1]$. Thus, 
$\theta=s\hat\Phi$ and $\psi=(1-s)\hat\Phi$, implying that
 $\hat\Phi$ is pure.

The proof of the converse is obviously similar and, thus, omitted.
\end{proof}

\begin{theorem}\label{pure thm} For every Hilbert space $\H$,
a matrix $T\in \left(C(S^1)^{(n)}\omin\B(\H)\right)_+$ is pure if and only if there exists 
a complex number $\lambda\in S^1$, a rank-1 projection $Q\in\B(\H)$, and a scalar $\alpha\geq0$ such that
\[
T = T_n(\lambda^{-1})\otimes (\alpha Q).
\]
\end{theorem}

\begin{proof} Suppose that $T=[\tau_{k-j}]_{k,j=1}^n\in \left(C(S^1)^{(n)}\omin\B(\H)\right)_+$ is pure, 
and express $T$ in its canonical tensor representation with respect
to the canonical linear basis of $C(S^1)^{(n)}$:
\[
T=\sum_{\ell=-n+1}^{n-1} r_\ell\otimes\tau_\ell.
\]
By the Duality Theorem, Theorem \ref{toeplitz duality}, each matrix $r_\ell$ induces a linear functional $\varphi_{r_\ell}$
on 
$C(S^1)_{(n)}$ via the action
\[
\varphi_{r_\ell}(f)=\hat f(-\ell),
\]
for each $f\in C(S^1)_{(n)}$.
The positivity condition \eqref{e:duality positivity gen} and  
Proposition \ref{pure1}
indicate that $\hat T:C(S^1)_{(n)}\rightarrow\B(\H)$
is a pure completely positive linear map, where the action of $\hat T$ on $C(S^1)_{(n)}$ is given by
\[
\hat T(f) = \sum_{\ell=-n+1}^{n-1}\varphi_{r_\ell}(f)  \tau_\ell = \sum_{\ell=-n+1}^{n-1} \hat f(-\ell)\tau_\ell.
\]
The operator system $C(S^1)_{(n)}$ is an operator subsystem of the unital C$^*$-algebra $C(S^1)$, and so
the pure completely positive linear map $\hat T$ has an extension to a pure completely 
positive linear map on $C(S^1)$ \cite[p.~180]{arveson1969}, which we denote again by $\hat T$.

The irreducible representations of the commutative C$^*$-algebra $C(S^1)$ are point evaluations. 
Because pure completely positive maps of C$^*$-algebras into $\B(\H)$ have Stinespring dilations \cite{stinespring1955}
in which the dilating 
homomorphisms are irreducible representations \cite[Corollary 1.4.3]{arveson1969}, there is an 
element $\lambda\in S^1$ and a linear map $w:\H\rightarrow\mathbb C$ such that
\[
\hat T(g) = w^*g(\lambda)w = g(\lambda) w^*w, \mbox{ for every }g\in C(S^1).
\]
Because $w^*w$ is of rank-1, there exists a scalar $\alpha\geq 0$ and a rank-1 projection $Q\in\B(\H)$ 
such that $w^*w=\alpha Q$. Hence, 
\[
\hat T(g) = \alpha g(\lambda) Q, \mbox{ for every }g\in C(S^1).
\]
If $k\in\{-n+1,\dots,-1,0,1,\dots,n-1\}$, then the function $z\mapsto z^k$ is an element of $C(S^1)_{(n)}$ and
\[
\hat T(z^k) = \alpha \lambda^k Q.
\]
Thus,
\[
\alpha \lambda^k Q = \hat T(z^k) = \sum_{\ell=-n+1}^{n-1}\varphi_{r_\ell}(z^k)  \tau_\ell =\tau_{-k}.
\]
As the equation above holds for every choice of $k$, 
we deduce $\tau_\ell=\lambda^{-\ell}(\alpha Q)$ for every $\ell$, which implies that
\[
T=\sum_{\ell=-n+1}^{n-1} r_\ell\otimes\tau_\ell = \sum_{\ell=-n+1}^{n-1} \lambda^{-\ell}r_\ell\otimes(\alpha Q) 
= T_n(\lambda^{-1})\otimes (\alpha Q).
\]
Hence, the pure element $T$ has the form $T_n(\lambda^{-1})\otimes (\alpha Q)$, 
for some $\lambda\in S^1$, $\alpha\geq0$, and rank-1 
projection $Q$.

Conversely, suppose that $T=T_n(\lambda^{-1})\otimes (\alpha Q)$, for some $\lambda\in S^1$, $\alpha\geq0$, and rank-1 
projection $Q\in\B(\H)$. Then $\hat T_n:C(S^1)^{(n)}\rightarrow \B(\H)$ is given by $\hat T_n(f)=\alpha f(\lambda)Q$. 
Considered as a map on
$C(S^1)$, $\hat T_n$ has the form $\hat T_n(g)=w^*g(\lambda) w$, 
where $w:\H\rightarrow \mathbb C$ is a linear map such that $w^*w=\alpha Q$.
As such a map is pure in the cone $\mathcal C\mathcal P(C(S^1), \B(\H))$ \cite[Corollary 1.4.3]{arveson1969}, 
it is also pure in the cone
$\mathcal C\mathcal P(C(S^1)_{(n)}, \B(\H))$. 
Hence, as the purity of $\hat T$ implies the purity of $T$, this completes the proof.
\end{proof}

\begin{corollary}[Separation]\label{sep thm}
Every element of $\left(C(S^1)^{(n)}\omin\M_p(\mathbb C)\right)_+$ is separable.
That is, a matrix $T\in C(S^1)^{(n)}\omin\M_p(\mathbb C)$ is positive if and only if
there are $\lambda_1,\dots, \lambda_m\in S^1$ and positive $b_1,\dots,b_m\in \M_p(\mathbb C)$
such that
\[
T= \sum_{j=1}^m T_n(\lambda_j)\otimes b_j.
\]
\end{corollary}

\begin{proof} The operator system $C(S^1)^{(n)}$ and the Hilbert space
$\H=\mathbb C^p$ have finite dimension, and so the convex cone
$\left(C(S^1)^{(n)}\omin\M_p(\mathbb C)\right)_+$ is closed. Hence,
by \cite[Theorem 18.5]{Rockafellar-book},
each element $T$ of $\left(C(S^1)^{(n)}\omin\M_p(\mathbb C)\right)_+$ has the form
\[
T=\sum_{j=1}^m T_j,
\]
for some pure elements $T_j$ of 
$\left(C(S^1)^{(n)}\omin\M_p(\mathbb C)\right)_+$. Because each $T_j$ is separable, by Theorem \ref{pure thm}, so is $T$.
\end{proof}

A positive element $x$ in an operator system $\osr $ is \emph{strictly positive} if there exists a real number $\delta>0$ such that
$x-\delta e_\osr\in\osr_+$. With the aid of the separation theorem (Corollary \ref{sep thm}),  we 
arrive at an interesting result below (Corollary \ref{c1}): strictly positive $2\times 2$ Toeplitz matrices whose entries are $2\times 2$
Toeplitz matrices are separable in $\left(C(S^1)^{(2)}\omin C(S^1)^{(2)}\right)_+$. 
First, though, we note that strictly max-positive Toeplitz matrices with Toeplitz
entries are separable.

\begin{proposition}\label{l1} If $x$ is strictly positive in $C(S^1)^{(n)}\omax C(S^1)^{(m)}$, then $x$ is separable.
\end{proposition}

\begin{proof} There exist $N\in\mathbb N$ and positive matrices
$s=[s_{ij}]_{i,j=1}^N\in \M_N(C(S^1)^{(n)})$ and $t=[t_{ij}]_{i,j=1}^N\in \M_m(C(S^1)^{(m)})$
such that $x=\displaystyle\sum_{i,j=1}^N s_{ij}\otimes t_{ij}$ 
\cite[Lemma 2.7]{farenick--kavruk--paulsen--todorov2014}. The matrix $s$ can be expressed as
$s=\displaystyle\sum_{i,j=1}^N s_{ij}\otimes e_{ij}$, and each matrix $s_{ij}$ is given by 
$s_{ij}=\displaystyle\sum_{\ell=-n+1}^{n-1}\zeta^{(ij)}_\ell r_\ell$. Thus, 
$s=\displaystyle\sum_{\ell=-n+1}^{n-1} r_\ell\otimes s_\ell$, where 
$s_\ell=\displaystyle\sum_{i,j=1}^N\zeta^{(ij)}_\ell e_{ij}$.
By Theorem \ref{main result}, $s$ is separable in 
$\M_N(C(S^1)^{(n)})
= C(S^1)^{(n)}\omin \M_N(\mathbb C)$; hence, 
$s=\displaystyle\sum_{k=1}^p a_k\otimes b_k$, for some positive $a_k\in C(S^1)^{(n)}$ and positive $b_k=[\beta^{(k)}_{ij}]_{i,j=1}^N\in\M_N(\mathbb C)$.
Likewise, $t=\displaystyle\sum_{f=1}^q c_f\otimes d_j$, for some $c_f\in (C(S^1)^{(m)})_+$ and $d_f=[\delta^{(f)}_{ij}]_{i,j=1}^N\in\M_N(\mathbb C)_+$.
Thus,
\[
x=\displaystyle\sum_{i,j=1}^N s_{ij}\otimes t_{ij} 
= \displaystyle\sum_{k=1}^p\displaystyle\sum_{f=1}^q\left[\left(\displaystyle\sum_{i,j=1}^N \beta_{ij}^{(k)}\delta_{ij}^{(f)} \right)a_k \right]\otimes c_f.
\]
For a given $k$ and $f$, the scalar $\displaystyle\sum_{i,j=1}^N \beta_{ij}^{(k)}\delta_{ij}^{(f)}$ is obtained by applying a positive linear functional on $\M_N(\mathbb C)$
to the Schur-Hadamard product $b_k\circ d_f$ of $b_k$ and $d_f$. Because $b_k\circ d_f$ is positive, this numerical value is a nonnegative real number. Hence, the equation above
represents $x$ as a sum of elementary tensors in which 
the tensor factors are positive Toeplitz matrices. 
\end{proof} 

\begin{corollary}\label{c1} Every strictly positive matrix in $C(S^1)^{(2)}\omin C(S^1)^{(2)}$ is separable.
\end{corollary}

\begin{proof} Because $\left(\osr\omin\oss\right)^d=\osr^d\omax\oss^d$ for any
finite-dimensional operator systems $\osr$ and $\oss$
\cite[Proposition 1.16]{farenick--paulsen2012},
the dual form of \cite[Theorem 4.7]{farenick--kavruk--paulsen--todorov2014} asserts that
every strictly positive element of $C(S^1)^{(2)}\omin C(S^1)^{(2)}$ is strictly positive in 
$C(S^1)^{(2)}\omax C(S^1)^{(2)}$. Hence, by Lemma \ref{l1}, every 
strictly positive matrix in $C(S^1)^{(2)}\omin C(S^1)^{(2)}$ is separable.
\end{proof}

\section{Entanglement}

We now consider the situation whereby the operator system $C(S^1)^{(n)}$ is
replaced by its dual $C(S^1)_{(n)}$ when tensoring with $\B(\H)$.

It is sometimes convenient to write the canonical linear basis of $C(S^1)_{(n)}$ as
a set of functions
$\chi_k:S^1\rightarrow\mathbb C$ defined by $\chi_k(z)=z^k$, for $k=-n+1,\dots,n-1$, where
$\chi_0$ is the Archimedean order unit for $C(S^1)_{(n)}$. Thus, an arbitrary element 
$F\in C(S^1)_{(n)}\omin\oss$ has the form
\begin{equation}\label{e:tensor}
F=\sum_{\ell=-n+1}^{n-1}\chi_\ell\otimes \tau_\ell,
\end{equation}
for some $\tau_\ell\in\oss$. Alternatively, the operator system $C(S^1)_{(n)}\omin\oss$ 
coincides with the operator system of all continuous functions
$F:S^1\rightarrow\oss$ of the form
\begin{equation}\label{e:function}
F(z)=\sum_{\ell=-n+1}^{n-1}z^\ell  \tau_\ell, \mbox{ for }z\in S^1,
\end{equation}
for some $\tau_\ell\in\oss$.
The latter viewpoint is particularly attractive for $\oss=\M_p(\mathbb C)$, for in such 
cases $F$ arises as a $p\times p$ matrix of trigonometric polynomials.

Although \eqref{e:tensor} and \eqref{e:function} give rise to equivalent forms of positivity, for the clarity of the
proofs below, it is useful to maintain a formal notational distinction between the two.
To this end, recall
that the linear functional $\varphi_{\chi_\ell}$ on $C(S^1)^{(n)}$ induced by $\chi_\ell\in C(S^1)_{(n)}$ has
the property that
\[
\varphi_{\chi_\ell}(r_k)
\;=\; \left\{
       \begin{array}{lcl}
           1   &:\quad&      \mbox{if }k=-\ell  \\
            0          &:\quad&       \mbox{if }k \not=-\ell
      \end{array}
      \right\}.
\]

\begin{lemma}\label{pure2} If $\oss$ is any operator system, then an element
$F\in \left(C(S^1)_{(n)}\omin \mathcal{S}\right)_+$ is pure if and only if $\hat{F}:C(S^1)^{(n)}\rightarrow\oss$ is a
pure completely positive linear map.
\end{lemma}

\begin{proof} The proof is similar to the proof of Lemma \ref{pure1}, where the role of $r_\ell$ in Lemma \ref{pure1}
is played by $\chi_\ell$ here.
\end{proof}

The following result describes the structure of pure elements in the positive cone of 
$ C(S^1)_{(n)}\omin\B(\H)$.
 
\begin{proposition}\label{ent1} Consider
$F\in \left(C(S^1)_{(n)}\otimes\B(\H)\right)_+$ 
in the function form \eqref{e:function}. If $F$
is pure, then 
there exists an operator $w:\H\rightarrow\mathbb C^n$ such that
\[
F(z)=w^*T_n(z^{-1})w,
\]
for every $ z\in S^1$.
\end{proposition}

\begin{proof}
Suppose
$F=\displaystyle\sum_{\ell=-n+1}^{n-1}\chi_\ell\otimes \tau_\ell$ is positive and pure. Thus, by Lemma \ref{pure2}, the linear map 
$\hat F:C(S^1)^{(n)}\rightarrow \mathcal{S}$ given by
$\hat F=\displaystyle\sum_{\ell=-n+1}^{n-1}\varphi_{\chi_\ell}\tau_\ell$ is completely positive and pure.
By the Arveson pure extension theorem \cite[p.~180]{arveson1969}, $\hat F$ extends to a pure completely positive linear map of $\M_n(\mathbb C)$ into $\B(\H)$,
which we continue to denote by $\hat{F}$. Since all irreducible representations of the C$^*$-algebra $\M_n(\mathbb C)$ are unitarily equivalent to the identity representation, 
the pure completely positive linear map $\hat F$ necessarily has the form
$\hat{F}(x)=w^*xw$, for all $x\in\MnC$, for some linear operator $w:\H\rightarrow\mathbb{C}^n$ \cite[Corollary 1.4.3]{arveson1969}. Thus,
 for each $k\in\{-n+1,\dots,-1,0,1,\dots,n-1\}$,
\[
\hat{F}(r_k)=w^*r_kw=\sum_{\ell=-n+1}^{n-1}\varphi_{\chi_\ell}(r_k)\tau_\ell=\tau_{-k}\,.
\] 
Hence, considering the function form \eqref{e:function} of $F$, we deduce that
\[
F(z)=\sum_{\ell=-n+1}^{n-1}z^\ell (w^*r_{-\ell}w)=w^*\left(\sum_{\ell=-n+1}^{n-1}z^\ell  r_{-\ell}\right)w =  w^*T_n(z^{-1})w,
\]
for all $z\in S^1$.
\end{proof}

Because of finite dimensionality, the cone
$\left(C(S^1)_{(n)}\otimes\M_p(\mathbb C)\right)_+$ closed. Hence, we have the following consequence:

\begin{corollary} A matrix $F\in C(S^1)_{(n)}\omin\M_p(\mathbb C)$ is positive if and only if there
are linear maps $w_j:\mathbb C^p\rightarrow\mathbb C^n$, for $j=1,\dots,m$, such that
\[
F(z)=\sum_{j=1}^m w_j^*T_n(z)w_j,
\]
for $z\in S^1$.
\end{corollary}

Our main result in this section is the following one: the universal Toeplitz matrices are entangled.

\begin{theorem}\label{ent T_n} For every $n\geq2$, the universal Toeplitz matrix
$T_n$ is pure and entangled in $\left(C(S^1)_{(n)}\omin\M_n(\mathbb C)\right)_+$.
\end{theorem}

\begin{proof}
Suppose, on the contrary, that $T_n$ is separable. Thus, using the form \eqref{e:function},
there are nonzero $f_1,\dots, f_m \in \left(C(S^1)_{(n)}\right)_+$ 
and nonzero $b_1, \dots, b_m\in\MnC_+$ such that $T_n(z)=\displaystyle\sum_{j=1}^m f_j(z)b_j$, for all $z\in S^1$. 
Let 
\[
\mathcal{Z}=\lbrace z_0\in S^1 \vert f_j(z_0)=0, \mbox{ for some } j=1, ..., m\rbrace .
\] 
Because $g_j(z)=z^{n-1}f_j(z)$ is a complex polynomial of degree at most $2n-2$, 
there are, for each $j$, at most finitely many $z\in S^1$ for which $f_j(z)=0$. Thus, the set $\mathcal{Z}$ is finite.

Select $z\in S^1\setminus \mathcal{Z}$. Because $T_n(z)$ is rank 1 and positive, 
there exist orthonormal vectors $\xi_1^z, ...,\xi_{n-1}^z \in \mathbb{C}^n$ such that $T_n(z)\xi_k^z=0$ for all $k=1, ..., n-1$. Thus,
\[
0=\langle T_n(z)\xi_k^z, \xi_k^z\rangle = \sum_{j=1}^m f_j(z)\langle b_j\xi_k^z, \xi_k^z\rangle.
\]
Since $f_j(z)>0$, for each $j$, it must be that $\langle b_j\xi_k^z, \xi_k^z\rangle=0$ and, hence, $b_j\xi_k^z=0$ for each $k$ and $j$,
by the positivity of $b_j$. 
Thus, $\ker{T_n(z)} \subseteq \ker{b_j}$ for all $j$, and so $\ran T_n(z)\supseteq\ran b_j$. Therefore, as $b_j\neq 0$ and $T_n(z)$ has rank 1, we have the equality
$\ran T_n(z)=\ran b_j$ for each $j$, and such is true for all $z\in S^1\setminus\mathcal{Z}$. Hence, 
\[
\ran b_j=\ran T_n(z)=
\left\{\delta\begin{array}{c|c}
\begin{bmatrix}
1 \\
z^2 \\
\vdots \\
z^{n-1} \\
\end{bmatrix} &
\delta\in\mathbb{C}
\end{array}\right\},
\]
for all $z\in S^1\setminus\mathcal{Z}$ and all $j$. Because $S^1\setminus\mathcal{Z}$ is infinite, 
there exist $z\in S^1$ such that $z, -z \notin \mathcal{Z}$. With such a complex number $z$,
\[
\begin{bmatrix}
1 \\
z \\
z^2 \\
z^3 \\
\vdots \\
z^{n-1} \\
\end{bmatrix}
\mbox{ and } 
\begin{bmatrix}
1 \\
-z \\
z^2 \\
-z^3 \\
\vdots \\
(-z)^{n-1} \\
\end{bmatrix}
\]
are linearly independent and belong to the range of each $b_j$. Hence, 
$\text{dim} (\ran b_j)\geq 2$, in contradiction to $\text{dim} (\ran b_j)=1$. 
Therefore, it cannot be that $T_n$ is separable.

To show that $T_n$ is pure, we need only show that the completely positive linear map
$\hat T_n:C(S^1)^{(n)}\rightarrow\M_n(\mathbb C)$ is pure (by Lemma \ref{pure2}). For each $k$,
$\hat T_n(r_k)=r_{-k}=u_n^*r_ku_n$, where $u_n\in\M_n(\mathbb C)$ is the unitary matrix that implements
the transpose map on the Toeplitz operator system. Hence, $\hat T_n$ is the restriction 
to $C(S^1)^{(n)}$
of the irreducible representation 
$\pi:\M_n(\mathbb C)\rightarrow\M_n(\mathbb C)$ given by $\pi(x)=u_n^*xu_n$. Because
$C(S^1)^{(n)}$ generates the C$^*$-algebra $\M_n(\mathbb C)$, and because $\pi$ is a boundary
representation for $C(S^1)^{(n)}$ \cite[Theorem 2.1.1 and Remark 2 on p.~288]{arveson1972}, 
the restriction $\hat T_n$ of $\pi$ to $C(S^1)^{(n)}$ is necessarily
pure \cite[Proposition 2.12]{farenick--tessier2022}.
\end{proof}


\section{Complete Positivity}

The innovation introduced by Stinespring in his classic paper \cite{stinespring1955} was to focus on those
positive linear maps $\psi:\A\rightarrow\B(\K)$, where $\A$ is a unital C$^*$-algebra and $\K$ is a Hilbert space, 
for which $\psi^{(n)}=\mbox{\rm id}_{\M_n(\mathbb C)}\otimes\A$ is a positive linear map of
$\M_n(\mathbb C)\omin\A$ into $\M_n(\mathbb C)\omin\B(\K)$ for every $n\in\M_n$. With the analogy
encapsulated by the definition below, the
role of $\M_n(\mathbb C)$ is taken by $C(S^1)^{(n)}$.

\begin{definition} If $\osr$ is an operator system, then a
linear map $\psi:\osr\rightarrow\B(\K)$, for some Hilbert space $\K$, is \emph{Toeplitz completely positive} if, 
for every $n\in\mathbb N$ and matrix
\[
\left[ \begin{array}{cccc} 
\tau_0 & \tau_{-1}   &  \dots&   \tau_{-n+1} \\
\tau_1 & \ddots & \ddots & \vdots \\
\vdots & \ddots & \ddots & \tau_{-1}  \\
\tau_{n-1} &   \dots &  \tau_{1}&\tau_0
\end{array}
\right]
\in \M_n(\osr)_+,
\]
the matrix 
\[
\left[ \begin{array}{cccc} 
\psi(\tau_0) & \psi(\tau_{-1}) &    \dots&  \psi(\tau_{-n+1}) \\
\psi(\tau_1) & \ddots & \ddots& \vdots\\
\vdots & \ddots & \ddots & \psi(\tau_{-1})  \\
 \psi(\tau_{n-1}) & \dots &  \psi(\tau_{1})&\psi(\tau_0)
\end{array}
\right]
\]
is a positive operator on $\displaystyle\bigoplus_1^n\K$.
\end{definition}

Theorem \ref{main result} was used in \cite[Corollary 7.4]{farenick2021} to establish 
Toeplitz completely positivity of positive linear maps on unital 
nuclear C$^*$-algebras,
but this fact is also true for nuclear operator systems using 
the main result of \cite{han--paulsen2011} regarding nuclear maps and arguments 
similar to those used in \cite[Corollary 7.4]{farenick2021}. We omit further discussion of the details
of this straightforward proof.

\begin{theorem}\label{nuclear} 
Every positive linear map on a nuclear operator system
is Toeplitz completely positive.
\end{theorem}

Our aim here is to prove the following similar result.

\begin{theorem}\label{normal} If $\H$ is a Hilbert space of infinite dimension, then
every normal positive linear map on $\B(\H)$ is Toeplitz completely positive.
\end{theorem}

\begin{proof} For each $n\in\mathbb N$, consider $C(S^1)^{(n)}\omin\B(\H)$ 
in the topology inherited from the ultraweak topology of the von Neumann algebra
$\M_n(\mathbb C)\otimes\B(\H)$. Thus, if $\psi:\B(\H)\rightarrow\B(\K)$ is a normal
positive linear map, then $\psi^{(n)}=\mbox{\rm id}_{\M_n(\mathbb C)}\otimes\psi$ 
is continuous with respect to the ultraweak topology of $\M_n(\mathbb C)\otimes\B(\H)$.

By Theorem \ref{main result}, the extremal rays of $\left(C(S^1)^{(n)}\omin\B(\H)\right)_+$ are generated by 
elements of the form $T_n(\lambda^{-1})\otimes (\alpha Q)$, for some $\lambda\in S^1$, $\alpha\geq0$, and
rank-1 projection $Q$. 
As the set of all finite sums of elements of the form $\alpha Q$ coincides 
with all finite-rank positive linear operators on $\B(\H)$, and because this set is ultraweakly dense in
$\B(\H)_+$, the set 
\[
\left\{\sum_{j=1}^m T_n(\lambda_j^{-1})\otimes(\alpha_jQ_j)\,|\,m\in\mathbb N,\, \lambda_j\in S^1,\,\alpha_j\geq0,\,
Q_j\mbox{ is a rank-1 projection}\right\},
\]
is ultraweakly dense in $\left(C(S^1)^{(n)}\omin\B(\H)\right)_+$.  
Furthermore, the action of $\psi^{(n)}$ on a pure element $T_n(\lambda^{-1})\otimes (\alpha Q)$ is given by
\[
\psi^{(n)}\left[T_n(\lambda^{-1})\otimes (\alpha Q)\right] = T_n(\lambda^{-1})\otimes\alpha\psi(Q),
\]
which is positive in $\M_n(\mathbb C)\otimes\B(\K)$. Hence,  by ultraweak continuity,
$\psi^{(n)}$ is positive on $C(S^1)^{(n)}\omin\B(\H)$.
\end{proof}

\section{Factorisation and Universality are Equivalent}

Our purpose in this paper has been to underscore the role and use of Toeplitz duality in results pertaining to 
separability and entanglement. Nevertheless, it is useful to make note of the relationship between Ando's universality theorem
\cite{ando1970} and Gurvits' separation theorem \cite{gurvits2001,gurvits--burnam2002}. 
How these two results intersect is
described by Theorem \ref{factorisation thm} below.

Before turning to Theorem \ref{factorisation thm}, let us first observe that the universal Toeplitz matrix valued function
$T_n:S^1\rightarrow\M_n(\mathbb C)$ can be evaluated, via continuous functional calculus, at any unitary operator $u$, 
thereby producing a positive $n\times n$ Toeplitz matrix $T_n(u)$ whose entries are integral powers $u^\ell$ of $u$.

\begin{theorem}[Factorisation Theorem]\label{factorisation thm}
If $T=[\tau_{k-j}]_{k,j=1}^n\in C(S^1)^{(n)}\omin\B(\H)$ is a positive Toeplitz matrix, then there are a Hilbert space $\K$, a
unitary operator $u\in \B(\K)$, and a bounded linear operator $w:\K\rightarrow\H$ such that
\begin{equation}\label{e:fac}
T= (1_n\otimes w)^*T_n(u)(1_n\otimes w),
\end{equation} 
where
$1_n$ denotes the $n\times n$ identity matrix.
If, moreover, the unitary operator $u$ has finite spectum, then $T$ is 
separable in the positive cone of $C(S^1)^{(n)}\omin\B(\H)$.
\end{theorem}

\begin{proof} By Ando's theorem \cite{ando1970}, the positivity of $T$ implies that
there is a completely positive linear map 
$\phi:C(S^1)\rightarrow \B(\H)$ such that $\phi(\chi_\ell)=\tau_\ell$, for each $\ell=-n+1,\dots,n-1$, and where
$\chi_\ell(z)=z^\ell$.
By Stinespring's theorem \cite{stinespring1955}, there are a Hilbert space $\K$, a unital 
$^*$-representation $\pi:C(S^1)\rightarrow\B(\K)$ of $C(S^1)$, and a bounded linear operator
$w:\K\rightarrow\H$ such that $\phi(f)=w^*\pi(f) w$, for every $f\in C(S^1)$. Let $u\in\B(\K)$ be the
unitary operator given by $u=\pi(\chi_1)$. Thus,
\[
\tau_\ell = \phi(\chi_\ell) = w^* \pi(\chi_\ell) w= w^*\pi(\chi_1)^\ell w^*= w^*u^\ell w,
\]
for each $\ell$, and so
\[
T=\displaystyle\sum_{\ell=-n+1}^{n-1}r_\ell\otimes w^*u^\ell w 
=(1_n\otimes w)^*\left( \displaystyle\sum_{\ell=-n+1}^{n-1}r_\ell\otimes u^\ell\right) (1_n\otimes w),
\]
which is precisely the sought-for equation $T=(1_n\otimes w)^*T_n(u)(1_n\otimes w)$. 

Suppose now that $u\in\B(\K)$ has finite spectrum. By the spectral theorem, there are $\lambda_1,\dots,\lambda_m\in S^1$
and pairwise-orthogonal projections $q_1,\dots,q_m\in\B(\K)$ such that $u^\ell=\displaystyle\sum_{j=1}^m\lambda_j^\ell q_j$,
for all $\ell\in\mathbb Z$. Therefore,
\[
T_n(u)=\displaystyle\sum_{\ell=-n+1}^{n-1}r_\ell\otimes (\displaystyle\sum_{j=1}^{m} \lambda_j^\ell q_j) 
= \displaystyle\sum_{j=1}^{m}\displaystyle\sum_{\ell=-n+1}^{n-1} \lambda_j^\ell r_\ell\otimes q_j 
= \displaystyle\sum_{j=1}^{m} T_n(\lambda_j^{-1})\otimes q_j.
\]
Hence,
$
T=(1_n\otimes w)^*T_n(u)(1_n\otimes w)=\sum_{j=1}^m T_n(\lambda_j^{-1})\otimes w^*q_jw$,
which is a separable element of the positive cone of $C(S^1)^{(n)}\omin\B(\H)$.
\end{proof}

The proof of Theorem \ref{factorisation thm} shows that the factorisation given in \eqref{e:fac} of a positive Toeplitz matrix $T$
is a consequence of Ando's universality theorem. Conversely, if $T=[\tau_{k-j}]_{i,j=1}^n$ is a Toeplitz matrix of operators 
on $\H$ whose $(k,j)$ entry, for each $k,j$, is
$\tau_{k-j}=w^*u^{k-j}w$, for some unitary $u\in\B(\K)$ and operator $w:\K\rightarrow\H$, then $T$ is positive and there is a completely positive linear
map 
$\phi:C(S^1)_{(n)}\rightarrow \B(\H)$ such that $\phi(\chi_\ell)=\tau_\ell$, for each $\ell=-n+1,\dots,n-1$. The completely positive linear map $\phi$ in question
arises from the fact that $C(S^1)$ is the universal C$^*$-generated by a unitary operator, and so there is a unital $*$-homomorphism
$\varrho:C(S^1)\rightarrow\B(\K)$ such that $\varrho(z^\ell)=u^\ell$, for every $\ell\in\mathbb Z$; thus, $\phi$ is given by $\phi(z^\ell)=w^*\varrho(z^\ell) w$,
for $\ell=-n+1,\dots,n-1$. Hence, Ando's universality theorem is a consequence of the factorisation theorem (Theorem \ref{factorisation thm}).

Separation depends on the finiteness of spectrum. 
The approach of Gurvits in \cite{gurvits2001,gurvits--burnam2002} to his separation theorem 
(Theorem \ref{main result}) is to 
show, by an elegant algebraic argument,
that a positive Toeplitz matrix $T\in C(S^1)^{(n)}\omin\M_p(\mathbb C)$ admits a factorisation of the form 
$(1_n\otimes w)^*T_n(u)(1_n\otimes w)$ in which the unitary
$u$ has finite spectrum. This conclusion can also be reached by way of the functional-analytic methods of the
present paper.

\begin{proposition} If $T\in C(S^1)^{(n)}\omin\M_p(\mathbb C)$ is positive, then 
there exist $q\in\mathbb N$, a linear map $w:\mathbb C^p\rightarrow\mathbb C^q$, and a unitary
matrix $u\in \M_q(\mathbb C)$ such that $T$ admits a factorisation of the form
$T=(1_n\otimes w)^*T_n(u)(1_n\otimes w)$.
\end{proposition}

\begin{proof} By Corollary \ref{sep thm}, $T=\displaystyle\sum_{j=1}^m T_n(\lambda_j)\otimes a_j$, for some $\lambda_j\in S^1$ and
positive rank-1 matrices $a_j$. View the $m$-tuple $a=(a_1,\dots,a_m)$ as a positive operator-valued measure
on the Borel sets of the point space $X=\{1,\dots,m\}$. By Naimark's Dilation Theorem \cite[Theorem 4.6]{Paulsen-book}, 
there exist a Hilbert space $\K$, mutually orthogonal projections $p_1,\dots,p_m\in\B(\K)$ such that
$\displaystyle\sum_{j=1}^m p_j=1_{\mathcal K}$, and a linear operator 
$w:\mathbb C^p\rightarrow\K$ such that $a_j=w^*p_j w$, for each $j$.
The proof in \cite{Paulsen-book} of Naimark's Dilation Theorem is an application of 
the Stinespring Dilation Theorem to the finite-dimensional abelian C$^*$-algebra $C(X)$. By Stinespring's proof \cite{stinespring1955},
minimal Stinespring dilations
of completely positive linear maps on finite-dimensional C$^*$-algebras occur on Hilbert spaces of finite dimension. Hence, the dilating space
$\mathcal K$ has the form $\mathcal K=\mathbb C^q$, for some $q\in\mathbb N$, and the operator $w$ is a linear map 
$w:\mathbb C^p\rightarrow\mathbb C^q$ satisfying $a_j=w^*p_j w$, for each $j$. 

With the unitary matrix $u=\displaystyle\sum_{j=1}^m \lambda_j p_j\in M_q(\mathbb C)$, we obtain
\[
\begin{array}{rcl}
T&=& \displaystyle\sum_{j=1}^m T_n(\lambda_j)\otimes a_j  
= \displaystyle\sum_{j=1}^m  \displaystyle\sum_{\ell=-n+1}^{n-1} \lambda_j^\ell r_\ell\otimes w^*p_jw \\ && \\
&=& (1_n\otimes w)^*\left(\displaystyle\sum_{\ell=-n+1}^{n-1} r_\ell \otimes \left[\displaystyle\sum_{j=1}^m \lambda_j^\ell p_j\right] \right)
 (1_n\otimes w),
\end{array}
\]
which proves that $T=(1_n\otimes w)^*T_n(u)(1_n\otimes w)$, as desired.
\end{proof}

\section*{Acknowledgement} 
We wish to thank the referee for a very careful reading of the original manuscript and for making a comment that led us to discover
Corollary \ref{c1}.


\end{document}